\documentclass[12pt]{article}
\usepackage{fullpage}
\usepackage{amssymb}
\usepackage{amsthm}
\usepackage{graphicx}
\usepackage[all,2cell]{xy}
\newtheorem{theorem}{Theorem}[section]
\newtheorem{corollary}[theorem]{Corollary}

\newtheorem{lemma}[theorem]{Lemma}

\begin{document}

\title{Explicit Demazure character formula for negative dominant characters} 

\author{S. Senthamarai Kannan}
\maketitle

Chennai Mathematical Institute, Plot H1, SIPCOT IT Park, 

Siruseri, Kelambakkam  603103, Tamilnadu, India.

E-mail:kannan@cmi.ac.in.

\begin{abstract}
In this paper, we prove that for any semisimple simply connected algebraic
 group $G$, for any regular dominant character $\lambda$ of a maximal torus $T$ of $G$ and for any element $\tau$ in the Weyl group $W$, the character
$e^{\rho}\cdot char(H^{0}(X(\tau), \mathcal{L}_{\lambda-\rho}))$ is equal to 
the sum  $\sum_{w\leq \tau}char(H^{l(w)}(X(w),\mathcal{L}_{-\lambda}))^{*})$ of the characters of dual of the top cohomology modules on the Schubert varieties $X(w)$, $w$  running over all elements satisfying $w\leq \tau$. Using this result,
we give a basis of the intersection of the Kernels of the Demazure 
operators $D_{\alpha}$ using the sums of the characters of
$H^{l(w)}(X(w),\mathcal{L}_{-\lambda})$, where the sum is taken over
all elements $w$ in the Weyl group $W$ of $G$.
\end{abstract}


\section{Introduction}
The following notations will be maintained throughout this paper.

Let $\mathbb{C}$ denote the field of complex numbers. Let  
$G$ a semisimple, simply connected algebraic group 
over $\mathbb{C}$.  We fix a maximal torus $T$ of $G$ and 
let $X(T)$ denote the set of characters
of $T$. Let $W = N(T)/T$ denote the Weyl group of $G$ with respect to $T$.
Let $R$ denote the set of roots of $G$ with respect to $T$.

Let $R^{+}$ denote the set of positive roots. Let $B^{+}$ be the Borel 
sub group of $G$ containing $T$ with respect to $R^{+}$.
Let $S = \{\alpha_1,\ldots,\alpha_l\}$ denote the set of simple roots in
$R^{+}$, where $l$ is the rank of $G$. 
Let $B$ be the Borel subgroup of $G$ containing $T$ with respect to the 
set of negative roots $R^{-} =-R^{+}$.

For $\beta \in R^+$ we also
use the notation $\beta > 0$.  
The simple reflection in the Weyl group corresponding to $\alpha_i$ is denoted
by $s_{\alpha_i}$.  

Let $\mathfrak{g}$ denote the Lie algebra 
of $G$. Let $\mathfrak{h}$ be the Lie algebra of $T$. Let $\mathfrak{b}$ 
be the Lie algebra of $B$.

We have $X(T)\bigotimes \mathbb{R}=Hom_{\mathbb{R}}(\mathfrak{h}_{\mathbb{R}}, \mathbb{R})$, the dual of the real form of $\mathfrak{h}$.

The positive definite $W$-invariant form on $Hom_{\mathbb{R}}(\mathfrak{h}_{\mathbb{R}}, \mathbb{R})$ induced by the Killing form of the 
Lie algebra $\mathfrak{g}$ of $G$ is denoted by $(~,~)$. 
We use the notation $\left< ~,~ \right>$ to
denote $\langle \nu, \alpha \rangle  = \frac{2(\nu,
\alpha)}{(\alpha,\alpha)}$. 

%Let $sl_{2,\alpha}$ denote the $3$ dimensional Lie sub algebra of 

%$\mathfrak{g}$ generated by $x_{\alpha}$, and $y_{\alpha}$.

Let $\leq$ denote the partial order on $X(T)$ given by $\mu\leq \lambda$
if $\lambda-\mu$ is a non negative integral linear combination of simple 
roots. We also say that $\mu < \lambda$ if $\mu\leq \lambda$, and 
$\mu \neq \lambda$.

We denote by $X(T)^+$ the set of dominant characters of 
$T$ with respect to $B^{+}$. Let $\rho$ denote the half sum of all 
positive roots of $G$ with respect to $T$ and $B^{+}$.

We denote by $X(T)^{+}_{reg}$ by the set of all regular dominant characters 
of $T$.

For any simple root $\alpha$, we denote the fundamental weight
corresponding to $\alpha$  by $\omega_{\alpha}$. 

For $w \in W$, let $l(w)$ denote the length of $w$. We define the 
dot action by $w\cdot\lambda=w(\lambda+\rho)-\rho$.

Let $w_0 \in W$ denote the longest
element of the Weyl group.

For $w \in W$, let $X(w):=\overline{BwB/B}$ denote the Schubert variety in
$G/B$ corresponding to $w$.

Let $\leq$ denote the Bruhat order on $W$. We also say that 
$w < \tau$ if $w\leq \tau$, and  $w \neq \tau$.

When $\lambda$ is dominant,  we have $H^{i}(X(w), \mathcal{L}_{\lambda})=(0)$, 
for all $w\in W$ and for all $i\geq 1$. Further, Demazure character formula 
gives the character of the $T$-module $H^{0}(X(w), \mathcal{L}_{\lambda})$ for 
dominant characters $\lambda$ of $T$. 

Demazure character formula gives only the Euler characteristic of the 
line bundle $\mathcal{L}_{\lambda}$, when $\lambda$ is not dominant 
(see [7, II, 14.18])  or (see theorem in page 617, [1]).

However,  when $\lambda$ is not dominant, the character of the individual 
cohomology modules $H^{i}(X(w), \mathcal{L}_{\lambda})$ is not known 
explicitly.

It is a difficult problem to understand the explicit characters of the 
individual cohomology modules $H^{i}(X(w), \mathcal{L}_{\lambda})$ for an 
arbitrary $\lambda\in X(T)$. 

When $\lambda$ is a regular dominant character of $T$, we have 
$H^{i}(X(w), \mathcal{L}_{-\lambda})=(0)$ for all $i\neq l(w)$ 
(cf [2], corollary 4.1 ). Further, if $\langle \lambda, \alpha \rangle \geq 2$,   $H^{l(w)}(X(w), \mathcal{L}_{-\lambda})$ is a non zero highest weight 
$B$-module (see [8], theorem 5.8 (1)). 

So, we need to study the character of the $T$- module $H^{l(w)}(X(w),\mathcal{L}_{-\lambda})$ for any $w$, and for any regular dominant character $\lambda$ 
of $T$.

The aim of this paper is to give a formula for the character of  
$H^{l(\tau)}(X(\tau), \mathcal{L}_{-\lambda})$ for any regular dominant 
character $\lambda$ of $T$ and for any $\tau\in W$, inductively using the 
characters of $H^{l(w)}(X(w), \mathcal{L}_{-\lambda})$, where $w$ is running 
over elements of $W$ satisfying $w < \tau$.

More precisely, we prove that 

{\bf Theorem:}

{\it Let $\tau\in W$. Let $\lambda$ be a regular dominant character of $T$.

Then, the sum  $\sum_{w\leq \tau}char(H^{l(w)}(X(w),\mathcal{L}_{-\lambda}))^{*})$ is equal to $e^{\rho}\cdot char(H^{0}(X(\tau), \mathcal{L}_{\lambda-\rho}))$.}

Please see theorem (3.2) for a precise statement.
  
When $X(w)$ is Gorenstein, we relate two characters  $\psi_{w}$  
and  $\chi^{\prime}_{w}$ of $T$ which are involved in using Serre duality
on the Schubert Variety $X(w)$.

For a precise statement, see corollary (3.3).

The characters are discussed in [10] for arbitrary $G$, and in 
[11] when $G$ is of type $A_{n}$.

We give a basis of the intersections of the Kernels of the Demazure 
operators $D_{\alpha}$ using the sums of the characters of
$H^{l(w)}(X(w),\mathcal{L}_{-\lambda})$, where the sum is running over
$w\in W$.

For a precise statement, see corollary (3.4).

The organisation of the Paper is as follows:

Section 2 consists of Preliminaries. In Section 3, we prove the
Theorem stated above and derive some useful consequences. 

\section{Preliminaries}
\label{prelim}

We denote by $U$ (resp. $U^{+}$ ) the unipotent radical of $B$
(resp $B^{+}$). We denote by
$P_{\alpha}$ the minimal parabolic subgroup of $G$ containing $B$ and
$s_{\alpha}$.  Let $L_{\alpha}$ denote the Levi subgroup of
$P_{\alpha}$ containing $T$. We denote by $B_{\alpha}$ the
intersection of $L_{\alpha}$ and $B$. Then $L_{\alpha}$ is the product
of $T$ and a homomorphic image $G_{\alpha}$ of $SL(2)$ via a homomorphism
$\psi: SL(2)\longrightarrow L_{\alpha}$.
(cf. [7,II.1.1.4]). 

We refer to [6] for notation and preliminaries on semisimple Lie algebras
and their root systems.  

We choose an ordering of positive roots. Let $U_{\alpha}$ denote the $T$- 
stable one dimensional root subgroup of $G$ corresponding to a positive 
root $\alpha$.  We write $U^{+}$ as product 
$\Pi_{\alpha\in R^{+}}U_{\alpha}$ in the orderiong choosen as above.

For a fixed $w\in W$, the set of all positive roots $\alpha$ which are 
made negative by $w^{-1}$ by  $R^{+}(w^{-1})$. 

Given a $w \in W$ the closure in $G/B$ of the $B$ orbit of the coset
$wB$ is the Schubert variety corresponding to $w$, and is denoted by
$X(w)$.

For each $w\in W$, we can write the cell $C(w)$ in the 
Schubert variety $X(w)$ as product $\Pi_{\alpha\in R^{+}(w^{-1})}U_{\alpha}$
in the same ordering choosen as above.

Let $X_{\alpha}$ denote the coordinate function on $U^{+}$ correspoding to
the root subgroup $U_{\alpha}$ of $U^{+}$.

For any character $\chi$ of $B$, we denote by $\mathbb{C}_{\chi}$ the 
one dimensional representation of $B$ corresponding to $\chi$.

We make use of following points in computing cohomologies. 

{\it  Since $G$ is simply connected, the morphism 
$\psi: SL(2)\longrightarrow G_{\alpha}$ is an isomorphism, 
and hence $\psi:SL(2)\longrightarrow L_{\alpha}$  is injective.
We denote this copy of $SL(2)$ in $L_{\alpha}$  by $SL(2,\alpha)$ 
We denote by $B^{\prime}_{\alpha}$ the intersection of $B_{\alpha}$
and $SL(2,\alpha)$ in $L_{\alpha}$.

We also note that the morphism 
$SL(2, \alpha)/B_{\alpha}^{\prime}\hookrightarrow  L_{\alpha}/B_{\alpha}$ 
induced by $\psi$ is an isomorphism.}

{\it Since $L_{\alpha}/B_{\alpha}\hookrightarrow P_{\alpha}/B$ 
is an isomorphism, to compute the cohomology $H^{i}(P_{\alpha}/B, V)$
for any $B$- module $V$, we treat $V$ as a $B_{\alpha}$- module 
and we compute $H^{i}(L_{\alpha}/B_{\alpha}, V)$ }

 We recall some basic facts and results about Schubert
varieties. A good reference for all this is the book by Jantzen (cf
[7, II, Chapter 14 ]).

Let $w = s_{\alpha_{i_{1}}}s_{\alpha_{i_{2}}}\ldots s_{\alpha_{i_{n}}}$ be a reduced
expression for $w \in W$. Define 
\[
Z(w) = \frac {P_{\alpha_{i_{1}}} \times P_{\alpha_{i_{2}}} \times \ldots \times 
P_{\alpha_{i_{n}}}}{B \times \ldots
\times B},
\]
where the action of $B \times \ldots \times B$ on $P_{\alpha_{i_{1}}} \times P_{\alpha_{i_{2}}}
\times \ldots \times P_{\alpha_{i_{n}}}$ is given by $(p_1, \ldots , p_n)(b_1, \ldots
, b_n) = (p_1 \cdot b_1, b_1^{-1} \cdot p_2 \cdot b_2, \ldots
,b^{-1}_{n-1} \cdot p_n \cdot b_n)$, $ p_j \in P_{\alpha_{i_{j}}}$, $b_j \in B$. 
We denote by $\phi_w$ the birational surjective morphism
$
\phi_w : Z(w) \longrightarrow X(w)$.

We note that for each reduced expression for $w$, $Z(w)$ is smooth, however, 
$Z(w)$ may not be independent of a reduced expression.

Let $f_n : Z(w) \longrightarrow Z(ws_{\alpha_{n}})$ denote the map induced by the
projection $P_{\alpha_1} \times P_{\alpha_2} \times \ldots \times 
P_{\alpha_n} \longrightarrow P_{\alpha_1} \times P_{\alpha_2}
\times \ldots \times P_{\alpha_{n-1}}$. Then we observe that $f_n$ is a $ P_{\alpha_n}/B
\simeq {\bf P}^{1}$-fibration.

Let $V$ be a $B$-module. Let ${\mathcal L}_w(V)$ denote the pull back to $X(w)$ of the homogeneous vector bundle on $G/B$ associated to $V$. 
{\em By abuse of notation} we denote the pull back of ${\mathcal L}_w(V)$ to $Z(w)$ 
also by ${\mathcal L}_w(V)$, when there is no cause for confusion. 
%Then for each $w \in W$ we obtain by the
%standard method of associated construction an induced bundle ${\mathcal
%L}_w(V)$ on $Z(w)$. 
Then, for $i \geq 0$, we have the following isomorphisms of
$B$-linearized sheaves
\[
R^{i}{f_{n}}_{*}{\mathcal L}_w(V) = {\mathcal
L}_{ws_{\alpha_{n}}}(H^{i}(P_{\alpha_n}/B, {\mathcal L}_w(V)).
\]
This together with easy applications of Leray spectral sequences is
the constantly used tool in what follows. We term this the {\em
descending 1-step construction}. 

We also have the {\em ascending 1-step construction} which too is used
extensively in what follows sometimes in conjunction with the
descending construction. We recall this for the convenience of the
reader.

Let the notations be as above and write $\tau = s_{\gamma}w$, with 
$l(\tau) = l(w) +1$, for some simple root $\gamma$.  Then we
have an induced morphism
\[
g_1: Z(\tau) \longrightarrow P_{\gamma}/B \simeq {\bf P}^1,
\]
with fibres given by $Z(w)$. Again, by an application of the Leray
spectral sequences together with the fact that the base is a ${\bf
P}^1$, we obtain for every $B$-module $V$ the following exact sequence
of $P_{\gamma}$-modules:
$$
(0) \longrightarrow H^{1}(P_{\gamma}/B, R^{i-1}{g_{1}}_{*}{\mathcal L}_w(V)) 
\longrightarrow 
H^{i}(Z(\tau) , {\mathcal L}_{\tau}(V)) \longrightarrow
H^{0}(P_{\gamma}/B , R^{i}{g_{1}}_{*}{\mathcal L}_w(V) ) \longrightarrow (0).$$

This short exact sequence of $B$-modules will be used frequently in this paper.
So, we denote this short exact sequence by {\it SES} when ever this is being 
used.

We also recall the following well-known isomorphisms:
\begin{itemize}

\item ${\phi_w}_*{\mathcal O}_{Z(w)} = {\mathcal O}_{X(w)}$.

\item $R^{q}{\phi_w}_*{\mathcal O}_{Z(w)} = 0$ for $q > 0$.

\end{itemize}

This together with [7, II. 14.6]  implies that we may use the
Bott-Samelson schemes $Z(w)$ for the computation and study of all the
cohomology modules $H^{i}(X(w) , {\mathcal L}_w(V))$. Henceforth in this paper 
we shall use the Bott-Samelson schemes and their
cohomology modules in all the computations.

{\it Simplicity of Notation}
If $V$ is a $B$-module and ${\mathcal L}_w(V)$ is the induced
  vector bundle on $Z(w)$ we denote the cohomology modules
  $H^{i}(Z(w) , {\mathcal L}_w(V))$ by $H^{i}(w ,V)$.

In particular if $\lambda$ is a character of $B$ we 
denote the cohomology modules $H^{i}(Z(w) , {\mathcal L}_{\lambda} )$ by
$H^i(w, \lambda)$. 

\subsubsection {Some constructions from Demazure's paper}
We recall briefly two exact sequences from [4] that Demazure used in his short
proof of the Borel-Weil-Bott theorem (cf. [3] ). We use the same 
notation as in [4]. 

Let $\alpha$ be a simple root and let $\lambda \in X(T)$ be a weight
such that $\langle \lambda , \alpha \rangle  \geq 0$. For such a 
$\lambda$, we denote by $V_{\lambda,\alpha}$ the module 
$H^0(P_{\alpha}/B, \mathcal{L}_{\lambda})$ . Let $\mathbb{C}_{\lambda}$ denote the one dimensional $B$- module.

Here, we recall the following lemma due to Demazure on a short exact sequence of $B$ - modules: (to obtain the second sequence we need to assume that
$\langle \lambda , \alpha \rangle \geq 2$).

\begin{lemma}

$$
\begin{array}{l}
\mbox{$(0) \longrightarrow K \longrightarrow V_{\lambda,\alpha} \longrightarrow \mathbb{C}_\lambda \longrightarrow (0)$}.\\
\mbox{$(0) \longrightarrow \mathbb{C}_{s_{\alpha}(\lambda)} \longrightarrow K \longrightarrow V_{\lambda-\alpha,\alpha}
\longrightarrow (0)$}.\\
\end{array}
$$
\end{lemma}
A consequence of the above exact sequences is the following crucial lemma,
a proof of which can be found in [4].  
\begin{lemma} 
\begin{enumerate}
\item Let $\tau = ws_{\alpha}$, $l(\tau) = l(w)+1$. If
$\langle \lambda , \alpha \rangle \geq 0$ then 
$H^{j}(\tau , \lambda) = H^{j}(w, V_{\lambda, \alpha})$ for all $j\geq 0$.
\item Let $\tau = ws_{\alpha}$, $l(\tau) = l(w)+1$. If
$\langle \lambda ,\alpha \rangle  \geq 0$, then $H^{i}(\tau , \lambda ) =
H^{i+1}(\tau , s_{\alpha}\cdot \lambda)$. Further, if
$\langle \lambda , \alpha \rangle  \leq -2$, then $H^{i}(\tau , \lambda ) =
H^{i-1}(\tau ,s_{\alpha}\cdot \lambda)$. 
\item If $\langle \lambda , \alpha \rangle  = -1$, then $H^{i}( \tau ,\lambda)$ vanishes for all $i\geq 0$ (cf. Prop 5.2(b), [7] ).
\end{enumerate}
\end{lemma}

\section{Character of $(H^{l(w)}(w,-\lambda))^{*}$}

In this section, we describe the character of the $T$-module 
$(H^{l(\tau)}(\tau, -\lambda))^{*}$ in terms of 
$e^{\rho}\cdot char(H^{0}(\tau, \lambda-\rho))$ and the sums of characters of 
$(H^{l(w)}(w,-\lambda))^{*}$, where the sum is taken over all elements $w$ 
of $W$ satisfying $w < \tau$.

So, in particular, we obtain the character of the $T$-module 
$H^{l(\tau)}(\tau, -\lambda)$. 

Let $D_{w}$ denote the boundary divisor of $X(w)$. For abuse of notation,
we denote the sheaf on $X(w)$ corresponding to the Weil divisor $D_{w}$
by $D_{w}$.
Choose a non zero section $s\in H^{0}(X(w),D_{w})$ such that the zero 
locus $Z(s)$ is $D_{w}$.

Let $\mu$ be a dominant character of $T$.

We now cosider the restriction map:

$res_{w}:H^{0}(w,\mu) \longrightarrow H^{0}(D_{w},\mathcal{L}_{\mu})$.

Let $Ker_w$ denote the kernel of $res_{w}$.

Let $\epsilon_{w}$ (resp. $\epsilon_{w}^{*}$ ) denote the character of the 
$T$- module $Ker_{w}$ (resp. $(Ker_{w})^{*}$).

Let $\tau\in W$.
Let $h^{0}(\tau, \mu)$  (resp.  $(h^{0}(\tau, \mu))^{*}$ ) 
denote the character of the $T$- module 
$H^{0}(\tau,\mu)$ (resp. $(H^{0}(\tau,\mu))^{*}$ ).  

We now prove the following lemma giving a description of the character 
$h^{0}(\tau, \mu)$:

\begin{lemma} For any $\tau\in W$ and for any dominant character $\mu$ of 
$T$, we have  
$\sum_{w \leq \tau} \epsilon_{w} = h^{0}(\tau, \mu)$.  
\end{lemma}

\begin{proof}
We fix a $\tau\in W$. 

Since the cell $C(\tau)=\Pi_{\alpha\in R^{+}(\tau^{-1})}U_{\alpha}$ is an open subset 
of $X(\tau)$, the restriction map $H^{0}(\tau, \mu)\longrightarrow H^{0}(C(\tau), \mathcal{L}_{\mu})$ is injective.  

Since  $C(\tau)$ is the affine space $\Pi_{\alpha\in R^{+}(\tau^{-1})}U_{\alpha}$, the restriction of the line bundle $\mathcal{L}_{\mu}$ to $C(\tau)$ is trivial.  

Thus, $H^{0}(\tau, \mu)$ is a subspace of $\mathbb{C}[X_{\alpha}: \alpha\in R^{+}(\tau^{-1})]$.

So, every section $s \in H^{0}(\tau, \mu)$ associates a polynomial $f_{s}$
in the variables $\{X_{\alpha}: \alpha\in R^{+}(\tau^{-1})\}$.

Now, for every section $s\in H^{0}(\tau, \mu)$, there is a   
 unique minimal element $w\leq \tau$ with respect to
the Bruhat order on $W$ such that the restriction $f_{s}$ of $s$ to 
$C(\tau)$ is a polynomial in the variables 
$\{X_{\alpha}: \alpha\in R^{+}(w^{-1})\}$.

Hence, for every $s\in H^{0}(\tau, \mu)$, there is a unique minmal 
$w\leq \tau$ such that $s\in Ker_{w}$.
 
Thus, the character of the $B$- module $H^{0}(\tau, \mu)$ is equal to the 
sum $\bigoplus_{w\leq \tau}\epsilon_{w}$ of the characters of the 
dual modules $Ker_{w}$, $w$ running over all elements of $W$ 
satisfying $w\leq \tau$.

Hence, we have
$\sum_{w \leq \tau} \epsilon_{w} = h^{0}(\tau, \mu)$. This completes the proof
of lemma. 

\end{proof}

Let $\lambda$ be a regular dominant character of $T$. That is $\lambda$
satisfies $\langle \lambda, \alpha \rangle \geq 1$ for each simple 
root $\alpha$. 

We recall notation from section 2:
For $w\in W$, we denote by $H^{l(w)}(w,-\lambda)$ the 
top cohomology of the line bundle $\mathcal{L}_{-\lambda}$ on $X(w)$ 
associated to $-\lambda$. 

Let $e^{\rho}$ denote the element of the representation ring $\mathbb{Z}[X(T)]$
of $T$ corresponding to $\rho$. 
Here, we use exponential notation 
$e^{\rho}$ for using multiplication in the ring $\mathbb{Z}[X(T)]$.

Let $h^{l(w)}(w,-\lambda)$ denote the character of the $T$- module 
$H^{l(w)}(w,-\lambda)$.  

Let $(h^{l(w)}(w,-\lambda))^{*}$ denote the character of the dual 
$H^{l(w)}(w, -\lambda)^{*}$ of the $T$- module $H^{l(w)}(w,-\lambda)$.  

We have the following theorem:

\begin{theorem} For any $\tau\in W$, we have 
$\sum_{w \leq \tau} (h^{l(w)}(w,-\lambda))^{*} = e^{\rho}\cdot
h^{0}(\tau, \lambda-\rho)$.  
\end{theorem}
\begin{proof}
Let $w\in W$ be such that $w\leq \tau$.

Let $\omega_{X(w)}$ denote the dualising sheaf on $X(w)$.

{\it Observation 1 :}

By Serre duality, the $B$ modules $(H^{l(w)}(w , -\lambda))^{*}$  
and $H^{0}(w, \mathcal{L}_{\lambda}\otimes \omega_{X(w)})\bigotimes 
\mathbb{C}_{\psi_{w}}$ are isomorphic for some $\psi_{w}\in X(T)$. 
 
This character $\psi_{w}$ is discussed in [10].

On the other hand, we have $\omega_{X(w)}=-(\rho+D_{w})$,
where $D_{w}$ is the boundary of $X(w)$. So, we have 
$H^{0}(X(w), \mathcal{L}_{\lambda}\otimes \omega_{X(w)})=H^{0}(w, \lambda-\rho-D_{w})$

Choose a non zero section $s\in H^{0}(X(w), D_{w})$ such that the zero 
locus $Z(s)$ is $D_{w}$.

By excercise [5, II, 1.19 ], this section induces the following short exact 
sequence of sheaves on $X(w)$:

$$(0) \longrightarrow \mathcal{O}(-D_{w}) \longrightarrow 
\mathcal{O} \longrightarrow \mathcal{O}_{D_{w}} \longrightarrow 
(0).$$

Hence, we have the following exact sequence of $B$-modules:

$$(0) \longrightarrow H^{0}(w, \lambda-\rho-D_{w})\bigotimes \mathbb{C}_{\chi_{w}} \longrightarrow 
H^{0}(w,\lambda-\rho) \longrightarrow H^{0}(D_{w},\lambda-\rho) \longrightarrow (0).$$

Here, the $\mathbb{C}$ linear map $H^{0}(w, \lambda-\rho-D_{w})\otimes 
\mathbb{C}_{\chi_{w}} \longrightarrow H^{0}(w,\lambda-\rho)$ is induced by the 
multiplication by $s$, and the $\mathbb{C}$- linear map 
$H^{0}(w,\lambda-\rho) \longrightarrow H^{0}(D_{w},\lambda-\rho)$ is the restriction map, and we denote it by $res_{w}$. We note that $res_{w}$ is a 
homorphism of $B$- modules. 

Let $Ker_w$ denote the kernel of $res_{w}$.

From the above short exact sequence, we see that the character of the 
$B$- module $Ker_w$ is the same as that of  $H^{0}(w, \lambda-\rho-D_{w})
\bigotimes \mathbb{C}_{\chi_{w}}$.

Hence, the character of $H^{0}(w, \lambda-\rho-D_{w})$ is equal to 
$e^{-\chi_{w}}\cdot \epsilon_{w}$.

Using {\it Observation 1}, we see that the character of the $T$- module 
$(H^{l(w)}(w,-\lambda))^{*}$ is the same as that of 
$H^{0}(w, \lambda-\rho-D_{w}) \bigotimes \mathbb{C}_{\psi_{w}}$.

Thus, we have 

{\it Observation 2:} The character $h^{l(w)}(w,-\lambda)^{*}$ of 
$(H^{l(w)}(w,-\lambda))^{*}$ is equal to 
$e^{\psi_{w}-\chi_{w}}\cdot \epsilon_{w}$.

We now show that $\psi_{w}-\chi_{w}=\rho$.

Since $H^{l(w)}(w,-\lambda)$ is a highest weight $B$-module with highest
weight $w(-\lambda+\rho)-\rho$, its dual $(H^{l(w)}(w,-\lambda))^{*}$ 
is a lowest weight $B$-module with lowest weight $w(\lambda-\rho)+\rho$. 

On the other hand, $H^{0}(w,\lambda-\rho)$ is a lowest weight module with
lowest weight $w(\lambda)-w(\rho)$.

Hence, $Ker_{w}$ is a lowest weight module with lowest weight 
$w(\lambda)-w(\rho)$.

Using {\it Observation 2}, we have 

$w(\lambda-\rho)+\rho= w(\lambda)-w(\rho) + \psi_{w}-\chi_{w}$.

Hence, we have $\rho = \psi_{w}-\chi_{w}$.

Using {\it Observation 2}, we have

{\it Observation 3 :}

The character $h^{l(w)}(w,-\lambda)^{*}$ is equal to  
$e^{\rho}\cdot \epsilon_{w}$.

Now, taking $\mu=\lambda -\rho$ in  lemma(3.1), 
we have $\sum_{w\leq \tau} \epsilon_{w}
= h^{0}(\tau,\lambda-\rho)$.

Thus, using {\it Observation 3}, we see that 

$\sum_{w\leq \tau} (h^{l(w)}(w,-\lambda))^{*}= 
\sum_{w\leq \tau} e^{\rho}\cdot\epsilon_{w}= e^{\rho}\cdot h^{0}(\tau,\lambda-\rho)$. 

This completes the proof of the theorem.

\end{proof}

Let $w\in W$ be such that $X(w)$ is Gorenstein. 
Let $\chi^{\prime}_{w}$ be the character of $T$ such that 
$\mathcal{L}_{-2\rho+\chi^{\prime}_{w}}$  is the canonical line bundle on $X(w)$. 

When $G$ is of type $A_{n}$, $\chi^{\prime}_{w}$ is described in a nice 
combinatorial way in theorem 2, page 209 of [11]. 

Then, by Serre duality, there is a character $\psi_{w}$
of $T$ such that the $B$- modules $H^{l(w)}(w, -\lambda)^{*}$ and 
$H^{0}(w, \lambda-2\rho+\chi^{\prime}_{w})\bigotimes \mathbb{C}_{\psi_{w}}$
are isomorphic.

With notation as above, the following corollary relates the two characters 
$\chi^{\prime}_{w}$ and $\psi_{w}$ of $T$.

\begin{corollary}
Then, we have $\psi_{w}=\rho+w(\rho)-w(\chi^{\prime}_{w})$.
\end{corollary}
\begin{proof}

By Serre duality, the $B$ modules $(H^{l(w)}(w , -\lambda))^{*}$ and 
$H^{0}(w, \lambda-2\rho+\chi^{\prime}_{w})\bigotimes \mathbb{C}_{\psi_{w}}$
are isomorphic.

Hence, the lowest weights of these two $B$- modules  
are the same.

Thus, we have $w(\lambda-\rho)+\rho=w(\lambda-\rho)-w(\rho)+
w(\chi^{\prime}_{w})+\psi_{w}$.

Hence, we have $\psi_{w}=\rho+w(\rho)-w(\chi^{\prime}_{w})$.

\end{proof} 

Let $D_{\alpha}$ denote the Demazure operator on $\mathbb{Z}[X(T)]$
corresponding to a simple root $\alpha$.

We recall from [7, II, 14.17] that 

$$D_{\alpha}(e^{\lambda})= \frac{e^{\lambda}-e^{s_{\alpha}(\lambda)-\alpha}}{1 - e^{-\alpha}}.$$ 

Let $N_{\alpha}$ denote the kernel of $D_{\alpha}$.

Let $N$ denote the intersection $\bigcap_{\alpha\in S}N_{\alpha}$ of the 
kernels of all $D_{\alpha}$'s, $\alpha$ running over all simple roots.

Then, we have 

\begin{corollary}
$\{\sum_{w\in W} h^{l(w)}(w, -\lambda):\lambda\in X(T)^{+}_{reg} \}$ forms a 
$\mathbb{Z}$ basis for $N$.
\end{corollary}

\begin{proof}

We first show that $N$ consinsts of all $v\in \mathbb{Z}[X(T)]$
such that $D_{\alpha}(e^{\rho}v)=e^{\rho}v$ for each simple root 
$\alpha\in S$.

We fix a simple root $\alpha$. Let $v\in \mathbb{Z}[X(T)]$ be such that 
$D_{\alpha}(v)=0$.

Let $\mu\in X(T)$ be such that the coefficient of $e^{\mu}$ in the expression 
of $v$ is non zero but the coefficient of $e^{\mu^{\prime}}$ in the expression 
of $v$ is zero for every $\mu^{\prime} > \mu$ in the dominant ordering. 

Since $D_{\alpha}(v)=0$, using lemma(2.2), we see that 
either $\langle \mu, \alpha \rangle=-1$ and the coefficient of 
$e^{\mu-\alpha}$ in the expression of $v$ is zero or 
$\langle \mu, \alpha \rangle \geq 0$ and the coefficient of 
$e^{\mu-i\alpha}$ in the expression of $v$ is non zero for every 
$i=0,1,\cdots, \langle \mu, \alpha \rangle, 1 + \langle \mu, \alpha \rangle$. 

Let $t=\langle \mu, \alpha \rangle$. 

With out loss of generality, we may assume that the coefficient of $e^{\mu}$
in the expression of $v$ is a positive integer, say $a$. 

Proof of the case $t=-1$ is quite simple. So, we may assume that $t\geq 0$.

Using lemma(2.2), we see that $\sum_{i=0}^{t+1}e^{\mu-i\alpha}$ is in the 
Kernel of $D_{\alpha}$. Hence, we have $D_{\alpha}(v - a(\sum_{i=0}^{t+1}e^{\mu-i\alpha}))=0$. Hence, by induction on the dominant ordering on $X(T)$, we have 
$$D_{\alpha}(e^{\rho}(v - \sum_{i=0}^{t+1}e^{\mu-i\alpha}))=e^{\rho}(v - \sum_{i=0}^{t+1}e^{\mu-i\alpha}).$$ 

On the other hand, $D_{\alpha}(e^{\rho}(\sum_{i=0}^{t+1}e^{\mu-i\alpha}))
=e^{\rho}(\sum_{i=0}^{t+1}e^{\mu-i\alpha})$. 

Thus, we have $D_{\alpha}(e^{\rho}v)=e^{\rho}v$. 

Since $\alpha\in S$ was arbitrary, we must have $N=\{v\in \mathbb{Z}[X(T)]:
D_{\alpha}(e^{\rho}v)=e^{\rho}v ~ for ~ all ~ \alpha \in S\}$.

On the other hand,  $D_{\alpha}(v^{\prime})=v^{\prime}$ for every simple root
$\alpha$ if and only if $v^{\prime}$ is an integral linear combination 
of $h^{0}(w_{0}, \mu)$, $\mu$ running over dominant characters of $T$. 

Thus, we have 

{\it Observation 1 :}

The subset $\{e^{-\rho}h^{0}(w_{0}, \mu): \mu\in X(T)^{+}\}$ forms
a basis for $N$.

By theorem(3.2), we have  

$\sum_{w \in W} (h^{l(w)}(w,-\lambda))^{*} = e^{\rho}h^{0}(w_{0}, \lambda-\rho)$.  

The assertion of corollary follows from {\it Observation 1}.

\end{proof}

\end{document}